\documentclass{IEEEtran}
\usepackage{cite}
\usepackage{amsmath,amssymb,amsfonts}
\usepackage{algorithmic}
\usepackage{graphicx}
\usepackage{textcomp}
\usepackage{color}
\usepackage{cases}
\usepackage{empheq}
\usepackage{comment}
\newtheorem{theorem}{Theorem}

\newtheorem{definition}{Definition}
\newtheorem{remark}{Remark}
\newtheorem{proposition}{Proposition}
\newtheorem{proof}{Proof}
\newtheorem{corollary}{Corollary}
\newtheorem{lemma}{Lemma}
\newcommand{\dn}{\mathbf{d}}
\newcommand{\R}{\mathbb{R}}

\usepackage{caption}
\usepackage{subcaption}
\usepackage{hyperref}
\def\BibTeX{{\rm B\kern-.05em{\sc i\kern-.025em b}\kern-.08em
    T\kern-.1667em\lower.7ex\hbox{E}\kern-.125emX}}
\begin{document}
\title{ Finite/fixed-time Stabilization of Linear Systems  with States Quantization }
\author{Yu Zhou,  Andrey Polyakov,  and Gang Zheng, 
\thanks{Yu Zhou (yu.zhou@inria.fr), Andrey Polyakov (andrey.polyakov,inria.fr), Gang Zheng 
 (gang.zheng@inria.fr), are with Inria, Univ. Lille, CNRS, Centrale Lille, Parc scientifique de la Haute Borne 40, av. Halley, Bât A, Park Plaza 59650, Villeneuve d'Ascq, France. }
\thanks{}
\thanks{}}

\maketitle

\begin{abstract}
This paper develops a homogeneity-based approach to finite/fixed-time stabilization of linear time-invariant (LTI) system with quantized measurements. A sufficient condition for finite/fixed-time stabilization of multi-input LTI system under  states quantization is derived. It is shown that a homogeneous quantized state feedback with  logarithmic quantizer can guarantee finite/fixed-time stability of the closed-loop system provided that the quantization is sufficiently dense.  Theoretical results are supported with  numerical simulations.
\end{abstract}


\section{Introduction}
\label{sec:introduction}

Quantization and sampling are two main processes of digitization of a continuous signal.
Considering a continuous signal as a function of time,  the \textit{digitization of the time domain}
is known as sampling \cite{hristu2005handbook}, but the \textit{digitization of the co-domain} is referred to as quantization \cite{gray1998quantization}. Both methods are crucial to network control systems with a limited data transmission rate. In the last two decades, finite-time and fixed-time  state feedback stabilization  have attracted much attention due to their advantages, controlled settling time, faster convergence and better robustness (see, e.g., \cite{bhat_etal_2005_MCSS},\cite{Orlov2004:SIAM},  \cite{hong2006finite}, \cite{hong2010finite}, \cite{PolyakovTAC2012}, \cite{Song_etal2017:Aut}). 
It raises a  question: \textit{how to stabilize an LTI system in finite/fixed time in the case of digitization of a control law?} 

Continuous-time homogeneous systems with nonzero degrees are typical examples of dynamical systems with finite/fixed-time convergence rates \cite{bhat_etal_2005_MCSS}, \cite{andrieu2008homogeneous}, \cite{polyakov2020book}.  
The LTI system is a particular case of a homogeneous system with zero degree and an exponential convergence rate.
Continuous-time homogeneous systems have  many useful properties for control systems design: homogeneous Lyapunov functions, equivalence between local and global stability, tuning the finite/fixed-time convergence by means of  homogeneity degree, and small overshoots,  etc. (see, e.g., \cite{zubov1958systems}, \cite{rosier1992SCL},\cite{kawski1990_CTAT}, \cite{bhat_etal_2005_MCSS},  \cite{polyakov2020book}).   
Various homogeneous feedback control methods have been developed for continuous-time systems in the last two decades (see, e.g., \cite{grune2000SIAM}, \cite{hong2001output},\cite{andrieu2008homogeneous}, \cite{nakamura_etal_2009_TAC}, \cite{polyakov2019IJRNC}, \cite{zimenko_etal_2020_TAC} ). 
However, the digitization may destroy advantages of homogeneous controllers discovered for continuous-time case. Indeed, time sampling of the control signal generated by a homogeneous finite-time controller may invoke the so-called numerical chattering (see, e.g. \cite{AcaryBrogliato2010:SCL}, \cite{bernuau_etal2017Aut}). 
The problem of preservation of the finite/fixed-time convergence rate 
under sampled-in-time implementation of homogeneous controllers  has been solved recently \cite{polyakov_etal2023:Aut} using the concept of the so-called consistent discretization, while the problem of finite/fixed-time stabilization under state quantization still remains open. 

Mathematically, a  quantizer can be defined as a discrete-valued mapping, transforming continuous input signals into a discrete set of values.
The extensive research has been conducted on quantized control for LTI systems, including stabilization \cite{Fu_etal_2005_TAC},  \cite{delchamps1990TAC}, \cite{brockett_Liberzon2000TAC}, \cite{gao2008Aut} and robustness analysis \cite{kang2015Aut}, \cite{su2018Aut}, \cite{liberzon2007TAC}  as well as  a quantizer design \cite{elia_etal_2001_TAC}, \cite{bullo_Liberzon2006TAC},  \cite{brockett_Liberzon2000TAC}, \cite{wang2021TAC}, \cite{wang2022TAC}.
That research was focused on the asymptotic stability/stabilization of LTI system under quantization.
%The feedback law and the quantizer are critical components in quantization-based stabilization design. 
The discrete-time model is beneficial for quantization feedback law design for a linear system. However, the use of a discrete model for quantizer design in
the case of a nonlinear control system is challenging due to the
obvious difficulty of obtaining an exact nonlinear sampled-data
model. The majority of nonlinear feedback control laws that utilize quantized state measurements are developed under the so-called ``emulation" method \cite{nesic2009TAC}, i.e., the controller is first designed in continuous time, and next, the quantizer parameters are determined to guarantee the stability of the system.

The so-called logarithmic quantizer is derived using the quadratic Lyapunov function for a discrete-time linear time-invariant (LTI) system in \cite{elia_etal_2001_TAC}.
It follows the intuitive idea that the farther from the origin the state is, the less precise control action and the less precise information about the state are needed to stabilize the system. A linear control system with the logarithmic state quantization can preserve global exponential stability \cite{Fu_etal_2005_TAC}, \cite{bullo_Liberzon2006TAC}. Nonlinear control systems  with logarithmic quantizer are studied in \cite{ceragioli2007CDC}, \cite{liu2012Aut},  \cite{rehak2019IFAC}, \cite{yan2019Aut}, \cite{ye2022TAC}.
However, to the best of the author's knowledge, it remains the open question: \textit{is it possible to achieve finite/fixed-time stabilization of LTI system by means of a  feedback law with  a logarithmic state quantizer?}
In this paper, we positively answer this question and design a  homogeneity-based feedback preserving finite/fixed-time stability of a continuous-time LTI system in the case of a logarithmic quantization of the state measurements. The preliminary version of the paper has been presented at Conference on Decision and Control 2023 \cite{zhou_etal2023:CDC}, where only the controlled integrator chain has been studied. 

The key novel contribution of this paper is as follows:  a sufficient condition for finite/fixed-time stabilization  of LTI system by a homogeneous feedback with states quantization is established; finite/fixed-time stabilization of multi-input LTI system is realized by means of homogeneous feedback with logarithmic quantizer.

	The rest of the paper is organized as follows. The next section formulates the problem studied in the paper. 
In Section \ref{section:pre}, we briefly introduce homogeneous control systems analysis and homogeneous state feedback control design on linear system. Section \ref{section:main} presents the paper's main results, including the general stability condition of the homogeneous control system with quantized state measurements; fixed-time stabilization with logarithmic quantization. Finally, Section \ref{section:sim} presents numerical validation results.

\subsection*{Notations}\label{section:2}
$\mathbb{R}$ is the set of real numbers, $\mathbb{R}_{+}=\{x \in \mathbb{R}: x \geq 0\}$; $\mathbb{N}$ is set of all natural numbers without zero; $\mathbf{0}$ denotes the zero vector from $\mathbb{R}^{n}; \operatorname{diag}\left\{\lambda_{i}\right\}_{i=1}^{n}$ is the diagonal matrix with elements $\lambda_{i} ; P \succ 0(\prec 0, \succeq 0, \preceq 0)$ for $P \in \mathbb{R}^{n \times n}$ means that the matrix $P$ is symmetric and positive (negative) definite (semidefinite); $\lambda_{\min }(P)$ and $\lambda_{\max }(P)$ represent the minimal and maximal eigenvalue of a matrix $P=P^{\top}$; for $P \succeq 0$ the square root of $P$ is a matrix $M=P^{\frac{1}{2}}$ such that $M^{2}=P$.
$|x|$, $x \in \mathbb{R}$, is the absolute value of a scalar.  
$\|x\|=\sqrt{x^{\top}Px}$  is the weighted Euclidean norm of  $x\in \mathbb{R}^{n}$, where $P\succ 0$; $\|x\|_2=\sqrt{x^{\top}x}$ is canonical Euclidean norm of $x\in \mathbb{R}^{n}$;
 $\mathcal{K}$ is the set of continuous strictly increasing functions $\mathcal{K}: \mathbb{R}_+\mapsto\mathbb{R}_+$ such that $\mathcal{K}(0)=0$; $I_n$ denotes the identity matrix in $\mathbb{R}^n$. $\mathbb{S}(1)$ denotes the unit sphere $\mathbb{S}(1):=\{x\in\mathbb{R}^n:\|x\|=1\}$.
 
\section{Problem statement}\label{section:problem_sta}
Consider the linear system given by:
\begin{equation}\label{eq:lin_sys}
	\dot{x} = Ax + Bu,\quad t>0, \quad x(0)=x_0,
\end{equation}
where $x(t) \in \mathbb{R}^{n}$ is the system state, $u(t)\in \mathbb{R}^{m}$ is the control, $A \in \mathbb{R}^{n \times n}$ and $B \in \mathbb{R}^{n \times m}$ are system matrices. 
The pair $(A, B)$ is assumed to be controllable. Additionally, we assume  that all state measurements are subject to a quantization and the  quantized state vector is available for feedback control design.

Mathematically, the quantization involves partitioning the state space into disjoint sets. We define a quantization of $\mathbb{R}^n$ following ideas of  the  Voronoi diagram \cite{boots2009spatial}. 

Let $\mathcal{E}$ be a subset of $\mathbb{R}^n$ and $\Xi$ be a subset of $\mathbb{N}$. Let $\mathcal{D}_i\subseteq\mathcal{E}$ with $i\in\Xi$ be non-empty disjoint subsets covering $\mathcal{E}$: $\mathcal{D}_i\cap \mathcal{D}_j = \emptyset$ if $i\neq j$ and $\underset{i\in\Xi}{\cup} \mathcal{D}_i=\mathcal{E}$. Let $\mathcal{Q}=\{q_i\}_{i\in \Xi}$ such that $q_i\in \mathcal{D}_i$ for any $i\in \Theta$.

\begin{definition}	\itshape
	 A function $\mathfrak{q}:\mathcal{E}\mapsto \mathcal{Q}$ is said to be a quantization function (\textit{quantizer}) on the set $\mathcal{E}$ if $\mathfrak{q}(x)=q_i$ for all $x\in \mathcal{D}_i$ and all $i\in \Xi$. 
\end{definition}

The set $\mathcal{Q}$ represents a collection of all quantization values.
Each disjoint set $\mathcal{D}_i$ is called by a quantization cell, but  the  vector $\mathfrak{q}_i$ is the quantization seed of $\mathcal{D}_i$ which represents all states within the quantization cell in order to reduce transmission data (by sending just the number $i\in \Xi$ instead of the whole state vector $x\in \mathcal{D}_i$).
In this paper, we consider quantization functions defined on $\R^n$, i.e., $\mathfrak{q}:\mathbb{R}^n\rightarrow \mathcal{Q}$.
%In order to preserve the equilibrium of the system, 
% without losing generality, we introduce the following assumption:

%\begin{assumption}\label{asp:2}\itshape
%	Let us assume that $0\in \mathcal{D}_1$ and $q_1=\mathbf{0}$.
 %, i.e., $\mathfrak{q}(x)=\mathbf{0}$ for all $x\in\mathcal{D}_1$.
%\end{assumption}

The control system with quantized state measurements, denoted as:
\begin{equation}\label{eq:q_sys}
	\dot{x} = Ax+ Bu(\mathfrak{q}(x)),
\end{equation}
where $u(\cdot)$ is a state feedback law which uses the quantized measurements $q(x)$ of the state $x$. Such a closed-loop system is discontinuous. Its solutions are understood
in the sense of Filippov \cite{filippov2013differential}.

The objective of this paper is to design a controller  ensuring global finite/fixed-time stabilization\footnote{The system \eqref{eq:lin_sys} is said to be globally uniformly 
\begin{itemize}
\item 
\textit{Lyapunov stable} if there exists $\alpha\in \mathcal{K}$ such that $\|x(t)\|\leq \alpha(\|x_0\|), \forall x_0\in \R^n, \forall t\geq 0$;
\item \textit{finite-time stable} if it is globally uniformly Lyapunov stable and there exists a locally bounded function $T: \R^n \mapsto \R_+$ such that $\|x(t)\|=0, \forall t\geq T(x_0), \forall x_0\in \R^n$; 
\item fixed-time stable if it is globally uniformly finite-time stable and the settling-time function $T$ is globally bounded, i.e., $\exists T_{\max } \in \mathbb{R}_{+}$such that $T\left(x_0\right) \leq T_{\max }$, $\forall x_0 \in \mathbb{R}^n$.
\end{itemize}} of the system \eqref{eq:q_sys} using quantized state measurement.

\section{Preliminary}\label{section:pre}
 Homogeneity refers to a class of dilation symmetries, which have been shown to possess several useful properties for control design and analysis \cite{zubov1958systems}, \cite{khomenyuk1961systems}, \cite{kawski1991}, \cite{rosier1992SCL},  \cite{bhat_etal_2005_MCSS}. In finite-dimensional systems, linear dilation is widely used.
\begin{definition}\cite{kawski1991}
	\textit{A mapping $\dn(s):\mathbb{R}^n\mapsto \mathbb{R}^n$, $s\in\mathbb{R}$ is said to be a  dilation in $\mathbb{R}^n$ if }
	\begin{itemize}
		\item \textit{$\dn(0) = I_n$, $\dn(s+t) = \dn(s)\dn(t)=\dn(t)\dn(s)$, $\forall s, t\in\mathbb{R}$;}
		\item  \textit{$\lim\limits_{s\rightarrow -\infty}\|\dn(s)x\| =0$ and $\lim\limits_{s\rightarrow \infty}\|\dn(s)x\| =\infty$.}  
	\end{itemize}
\end{definition}
The dilation $\dn$ is continuous if the mapping $s\mapsto \dn(s)$ is continuous. 
Any \textit{linear} continuous dilation in $\R^n$ is given by
\begin{equation}\label{eq:dilation}
\dn(s) = e^{sG_\dn}:=\sum_{i=0}^{\infty} \tfrac{s^i G_\dn^i}{i!}.
\end{equation}
Since $\tfrac{d}{ds}e^{G_\dn s} = e^{G_\dn s} G_\dn= G_\dn e^{G_\dn s} $,  the derivative at $s=0$ is just an anti-Hurwitz  matrix $G_\dn\in \R^{n\times n}$ called the generator of a dilation.

\begin{definition}\label{def:monocity}
	\textit{A dilation $\dn$ is strictly monotone with respect to a norm $\|\cdot\|$ in $\R^n$ if $\exists \beta>0$ such that\vspace{-1mm}
	$$
	\|\mathbf{d}(s)\| \leq e^{\beta s}, \quad \forall s \leq 0.
	$$
}
\end{definition}
By default, below, we deal with linear continuous dilation given by \eqref{eq:dilation}.
The following result is the straightforward consequence of the quadratic Lyapunov function theorem for linear systems. 
\begin{proposition}\cite{polyakov2019IJRNC}\label{prop: 1}
	\textit{A dilation $\mathbf{d}$ is strictly monotone  with respect to the norm $\|z\|=\sqrt{z^\top P z}$ if and only if the following linear matrix inequality holds
	\begin{equation} P\succ 0, \
	P G_\dn+G_\dn^{\top} P \succ 0,\label{eq:mon_cond}
	\end{equation}
	where $G_\dn \in \mathbb{R}^{n\times n}$ is the generator of the linear continuous dilation $\mathbf{d}$ in $\R^n$. }
\end{proposition}

The linear continuous dilation $\dn$ induces an alternative topology in $\mathbb{R}^n$ via a ``homogeneous norm" \cite{kawski1995IFAC}.
%\begin{definition}
%	A continuous function $p: \mathbb{R}^n \rightarrow \mathbb{R}_{+}$ is said to be $\dn$-homogeneous norm if $p(x) \rightarrow 0$ as $x \rightarrow \mathbf{0}$ and $p(\dn(s) x)=e^s p(x)>0$ for $x \in \mathbb{R}^n \backslash\{\mathbf{0}\}$ and $s \in \mathbb{R}$.
%\end{definition}

\begin{definition}\cite{polyakov2019IJRNC}\label{def:hom_norm}
	\textit{The function $\|\cdot\|_\dn: \mathbb{R}^n \mapsto\R_+$ defined as $\|\mathbf{0}\|_{\dn}=0$ and 
	$$
	\|x\|_{\mathbf{d}}=e^{s}, \text { where } s \in \mathbb{R}:\left\|\mathbf{d}\left(-s\right) x\right\|=1, \quad x\neq \mathbf{0}
	$$
	is called the canonical homogeneous norm in $\mathbb{R}^n$, where $\mathbf{d}$ is a continuous monotone linear dilation on $\mathbb{R}^n$.}
\end{definition}

The monotonicity of the dilation is required to guarantee that the functional
$\|\cdot\|_\dn$ is single-valued and continuous at the origin.
 Although the canonical homogeneous norm does not satisfy the triangle inequality in $\R^n$, it is a norm in a special  Euclidean space \cite{polyakov2020book} being homeomorphic to $\R^n$. 
%The relation between the homogeneous norm and the Euclidean norm follows.
\begin{proposition}\cite{polyakov2019IJRNC}, \cite[Corollary 6.4]{polyakov2020book}\label{prop:3}
	\textit{Let $\mathbf{d}$ be strictly monotone linear continuous dilation on  $\mathbb{R}^n$. Then 
	\begin{equation}\label{eq:d_bound}
		\left\{
		\begin{aligned}
			&e^{\underline{\eta} s} \leq\lfloor\dn(s)\rfloor\le\|\mathbf{d}(s) \| \leq e^{\overline{\eta} s}, \ & s \geq 0, \\
			&e^{\overline{\eta} s} \leq \lfloor\dn(s)\rfloor\le \|\mathbf{d}(s) \| \leq e^{\underline{\eta} s}, \ &  s \leq 0,
		\end{aligned}
		\right. \quad \forall s\in \R,
	\end{equation}
	\begin{equation}\label{eq:hom_eclidean}
		\left\{
		\begin{aligned}
			\begin{aligned}
				& \|x\|_\dn^{\underline{\eta}} \leq\|x\| \leq \|x\|_\dn^{\overline{\eta}},&\ \|x\|\ge 1,\\
				& \|x\|_\dn^{\overline{\eta}} \leq\|x\| \leq \|x\|_\dn^{\underline{\eta}}, &\ \|x\|\le 1 ,	
			\end{aligned}
		\end{aligned}
		\right. \quad \forall x \in \mathbb{R}^n,
	\end{equation}
	where $\lfloor\dn(s)\rfloor=\inf _{u \in S}\|\mathbf{d}(s) u\|=\inf _{u \neq 0} \frac{\|\mathbf{d}(s) u\|}{\|u\|}$, \\
	$
	\overline{\eta}=\tfrac{1}{2} \lambda_{\max }\left(P^{\frac{1}{2}} G_{\mathrm{d}} P^{-\frac{1}{2}}+P^{-\frac{1}{2}} G_{\mathrm{d}}^{\top} P^{\frac{1}{2}}\right)>0,
	$\\
	$	\underline{\eta}=\tfrac{1}{2} \lambda_{\min }\left(P^{\frac{1}{2}} G_{\mathrm{d}} P^{-\frac{1}{2}}+P^{-\frac{1}{2}} G_{\mathrm{d}}^{\top} P^{\frac{1}{2}}\right)>0.
	$
}
\end{proposition}
The following lemma establishes a connection between the Euclidean and canonical homogeneous norms offering a valuable tool for analysis of  quantizers.
{
\begin{lemma}\label{lem:1}
	\textit{
		Let $G_\dn$ be a diagonalizable generator of a continuous linear dilation $\dn$ and $G_\dn = \Gamma \Lambda \Gamma^{-1}$, where $\Lambda$ is a diagonal matrix. 
		For any $z,y\in\mathbb{R}^n$, $0<\varepsilon<1$  there exists $\delta>0$ such that 
		\[
		|\xi_{z_i}|\leq \delta|\xi_{y_i}|, \forall i=1\ldots n \quad \Rightarrow \quad \|z\|_\dn\le\varepsilon \|y\|_\dn,
		\]
		where $\xi_z = \Gamma^{-1}z = [\xi_{z,1}, \xi_{z,2},\cdots,\xi_{z,n}]^\top$, $\xi_y = \Gamma^{-1}y = [\xi_{y,1}, \xi_{y,2},\cdots,\xi_{y,n}]^\top$,  the canonical homogeneous norm $\|\cdot\|_\dn$ is induced by the norm $\|x\|=\sqrt{x^\top P x}$, $x\in\mathbb{R}^n$ with a symmetric matrix $P\in \R^{n\times n}$ satisfying \eqref{eq:mon_cond}.
	}
\end{lemma}
\begin{proof}
    Notice that the $|\xi_{z,i}|\le\delta|\xi_{y,i}|$ indicates that
    $
\sum_{i=1}^{n}\xi_{z,i}^2 \le \delta^2\sum_{i=1}^{n}\xi_{y,i}^2$.
Besides, since $G_\dn$ is anti-Hurwitz, we have \vspace{-2mm}
        \[
\sum_{i=1}^{n}e^{-2\lambda_i s}\xi_{z,i}^2 \le \delta^2\sum_{i=1}^{n}e^{-2\lambda_i s}\xi_{y,i}^2, \ \forall s\in\mathbb{R}, \lambda_i >0,\vspace{-2mm}
    \]
and we note $\Lambda = \operatorname{diag}(\lambda_1,\lambda_2,\cdots, \lambda_n)\succ 0$.
The above inequality is equivalent to
     \[
\|e^{-\Lambda s}\Gamma^{-1}z\|_2\le \delta \|e^{-\Lambda s}\Gamma^{-1}y\|_2, \ s\in\mathbb{R}.
    \]
By Definition \ref{def:hom_norm} of the canonical homogeneous norm, we have $$
  1= \|\Gamma e^{-\Lambda s_y}\Gamma^{-1}y\|^2\!\geq\! \lambda_{\min}(\Gamma^\top P \Gamma) \|e^{-\Lambda s_y}\Gamma^{-1} y\|^2_2,$$ where $s_y=\ln \|y\|_{\dn}$. The latter inequality indicates that 
  \begin{equation*}
      \begin{aligned}
          \|e^{-\Lambda s_y}\Gamma^{-1}z\|_2&\le \tfrac{\delta}{\lambda_{\min}^{1/2}(\Gamma^\top P \Gamma)}.
      \end{aligned}
  \end{equation*}
 Then, since $G_\dn$ is anti-Hurwitz and $0<\varepsilon<1$, we can derive
\begin{equation}
    \begin{aligned}
        & \quad      \|e^{-\Lambda s_y} \Gamma^{-1}z\|_2\leq \tfrac{\sigma}{\lambda_{\min}^{1/2}(\Gamma^\top P \Gamma)}\\
        & \Rightarrow
        \|e^{-\Lambda\ln\varepsilon}e^{-\Lambda s_y} \Gamma^{-1}z\|_2\leq \tfrac{\sigma\lambda_{\max}(e^{-\Lambda\ln\varepsilon})}{\lambda_{\min}^{1/2}(\Gamma^\top P \Gamma)}\\
        &\Rightarrow
        \|\Gamma e^{-\Lambda\ln\varepsilon}e^{-\Lambda s_y} \Gamma^{-1}z\|\leq \tfrac{\sigma\lambda_{\max}(e^{-\Lambda\ln\varepsilon})\lambda_{\max}^{1/2}(\Gamma^\top P \Gamma)}{\lambda_{\min}^{1/2}(\Gamma^\top P \Gamma)}\\
        &
        \Rightarrow
        \|\dn(-\ln\varepsilon\|y\|)z\|\le \delta\lambda_{\max}(e^{-\Lambda\ln\varepsilon})\tfrac{\lambda_{\max}(\Gamma^\top P \Gamma)}{\lambda_{\min}(\Gamma^\top P \Gamma)}.
    \end{aligned}
\end{equation}
The latter inequality indicates that for a given $0<\varepsilon<1$, there exists a sufficiently small $\delta>0$ such that 
$$
\|\dn(-\ln\varepsilon\|y\|)z\|=\| \dn(\ln \tfrac{\|z\|_\dn}{\varepsilon\|y\|_{\dn}}) \dn(-\ln\|z\|_\dn)z\|\leq 1.$$
Taking into account $\|\dn(-\ln \|z\|_{\dn})z\|=1$, we obtain
$$
1\geq \|\dn(-\ln\varepsilon\|y\|_\dn)z\|\geq \left\lfloor\dn\left(\ln\tfrac{\|z\|_\dn}{\varepsilon\|y\|_\dn}\right)\right\rfloor.$$ 
Hence, using  \eqref{eq:d_bound}, we conclude 
  $\ln\tfrac{\|z\|_\dn}{\varepsilon\|y\|_\dn}\leq 0$. 
 The proof is complete.
	$\hfill \square$
\end{proof}
}

The homogeneous function and vector field are defined by following the paper  \cite{kawski1991}.
\begin{definition}
	\textit{A  vector field $f:\R^n\mapsto \R^n$ (resp., a function $h:\R^n \!\mapsto\! \R$) is said to be $\dn$-homogeneous of degree $\mu\!\in\! \R$ if
	\[
	f(\dn(s))=e^{\mu s} \dn(s) f(x), \quad \forall x\in\R^n, \quad \forall s\in \R,
	\]
	\[
	(\text{resp., } h(\dn(s))=e^{\mu s} h(x), \quad \forall x\in\R^n, \quad \forall s\in \R),
	\]
	where $\dn$ is a linear continuous dilation in $\R^n$.
}
\end{definition}

For a linear vector field $Ax$, its homogeneity is revealed in \cite{polyakov2019IJRNC} and \cite{zimenko_etal_2020_TAC}. The properties are listed as follows:

\begin{lemma}\cite{polyakov2019IJRNC,zimenko_etal_2020_TAC}\label{lem:nilpotent}
\textit{Let $\dn$ be a dilation. The linear vector field $Ax$ with $x\in\mathbb{R}^n$ and $A\in\mathbb{R}^{n\times n}$ is $\dn$-homogeneous of degree $\mu\in\mathbb{R}$ if and only if
\begin{equation}\label{eq:A}
AG_\dn = (\mu I + G_\dn)A,
\end{equation}
where $G_\dn\in\mathbb{R}^{n\times n}$ is a generator of $\dn$. Moreover, the following two claims are equivalent:}
\end{lemma}

\begin{itemize}
\item \textit{There exists a  dilation $\mathbf{d}(s)=e^{s G_{\mathbf{d}}}$ in $\mathbb{R}^n$ such that the vector field $x\mapsto Ax$ is $\mathbf{d}$-homogeneous of degree $\mu \neq 0$.}
\item \textit{A is nilpotent.}
\end{itemize}

A homogeneous system with a non-zero degree exhibits finite-time (negative) or nearly fixed-time stability \cite{bhat_etal_2005_MCSS}. A state feedback homogeneous stabilization on an LTI system has been proposed in \cite{zimenko_etal_2020_TAC} and \cite{polyakov2020book}.

\begin{theorem}\cite{zimenko_etal_2020_TAC, polyakov2020book}\label{thm:hom_linear}
	\textit{If the linear equation \vspace{-1mm}
	\begin{equation}\label{eq:G0}
		AG_0+BY_0=G_0A+A, \quad G_0B=\boldsymbol{0}.\vspace{-1mm}
	\end{equation}
	has a solution $G_0\in \R^{n\times n}$ and $Y_0\in \R^{m\times n}$ such that 
	$G_0-I_n$ is invertible and for any $\mu\in [-1,1/n]$ the matrix $G_{\dn}=I_n+\mu G_0$ is anti-Hurwitz, then the linear system \eqref{eq:lin_sys} with  the control \vspace{-1mm}
	\begin{equation}\label{eq:hom_con_P}
		u(x)\!=\!K_{0} x+\|x\|_{\dn}^{1+\mu}K \dn \left(-\ln \|x\|_{\dn}\right) x, \ K\!=\!Y X^{-1} \vspace{-1mm}
	\end{equation}
	is $\dn$-homogeneous of degree $\mu$ and globally asymptotically stable provided that 
	$K_{0}=Y_0(G_0-I_n)^{-1}$, $A_0=A+BK_0$ is nilpotent, $X \in \mathbb{R}^{n \times n}, Y \in \mathbb{R}^{m \times n}$ satisfy  the following algebraic system \vspace{-1mm}
	\begin{equation}
		\left\{\!\begin{aligned}
			&\!X A_{0}^{\top}\!+\!A_{0} X\!+\!Y^{\top} B^{\top}\!+\!B Y\!+\!\rho\!\left(X G_{\dn}^{\top}\!+\!G_{\dn} X\right)\!=\!\mathbf{0} \\
			&\!X G_{\dn}^{\top}+G_{\dn} X \succ 0, \quad X \succ 0
		\end{aligned}\right. \vspace{-1mm}
		\label{eq:LMI0}
	\end{equation}
	and  the canonical homogeneous norm $\|\cdot\|_{\dn}$ is induced by the norm $\|x\|\!=\!$ $\sqrt{x^{\top} X^{-1} x}$. 
	Moreover,  the closed-loop system is: globally uniformly finite-time stable for $\mu<0$; globally uniformly exponentially stable for $\mu=0$;  globally uniformly nearly fixed-time stable \footnote{The  system \eqref{eq:lin_sys} is globally uniformly nearly fixed-time stable if it is globally uniformly Lyapunov stable  and $\forall r>0, \exists T_r>0: \|x(t)\|<r, \forall t\geq T_r$ independently of $x_0\in \R^n$.} for $\mu>0$. 
}
\textit{Moreover, $\|\cdot\|_{\dn}$ is a Lyapunov function of the system:
	\begin{equation}\label{eq:norm_dt}
		\tfrac{d}{dt} \|x\|_{\dn}=-\rho\|x\|_{\dn}^{1+\mu}.	
	\end{equation}
}
\end{theorem}
If the pair $\{A,B\}$ is controllable, then any solution of the system \eqref{eq:G0} satisfies conditions of the latter theorem for $\mu\leq 1/k$, where $k$ is a minimal number such that $\text{rank}[B,AB,...,A^{k-1}B]=n$.
The feasibility of \eqref{eq:LMI0} has been proven in \cite{polyakov2016IJRNC}, \cite{zimenko_etal_2020_TAC}.
Notice also that the matrix $G_0$ (and respectively, $G_{\dn}$) in Theorem \ref{thm:hom_linear} can always be selected diagonalizable.

\section{Homogeneous Control with State Quantization}\label{section:main}

\subsection{Stability of homogeneous control system with quantized state measurements}
Due to nonlinearity, the design of finite/fixed-time controllers
with quantized state measurements is a non-trivial task. To tackle this problem we propose the  following two step procedure: 

\textit{Step 1}: First, we design a continuous-time closed-loop system with  finite/fixed-time stabilizing feedback. Next, we replace the continuous states with quantized states in the feedback law, and derive some sufficient finite/fixed-time stability conditions of the closed-loop system with quantization. 

\textit{Step 2}: Leveraging the derived quantization conditions, we design a quantization function using a logarithmic quantizer such that finite/fixed-time stabilization is preserved.

The continuous-time homogeneous finite/fixed-time stabilizer \eqref{eq:hom_con_P} is assumed to be designed for the LTI system  by means of  Theorem \ref{thm:hom_linear}. Our goal now is to derive a  sufficient condition for the quantization function to preserve finite/fixed stability of the closed-loop system in the case of quantized state measurements. The homogeneous controller  \eqref{eq:hom_con_P} contains two terms. The linear one $K_0x$ is aimed at homogenization of $A_0=A+BK_0$ with non-zero degree (see Lemma \ref{lem:nilpotent}), while the nonlinear term stabilizes the closed-loop homogeneous system. If $A$ is a nilpotent matrix, then $K_0$ may be equal to zero. Below we study both cases $K_0\neq\mathbf{0}$ and $K_0=\mathbf{0}$, since, in the latter case, the control law also becomes a homogeneous function and restrictions to the quantizer can be relaxed.

\begin{theorem}\label{thm:main}
%	\Andrey{Let all parameters of the homogeneous feedback \eqref{eq:hom_con_P} be designed as in Theorem \ref{thm:hom_linear}. Let the quantized feedback be defined as ..., where $q$....    If $K_0=0$, ... then }
	\textit{ Let all parameters of the homogeneous feedback \eqref{eq:hom_con_P} be designed as in Theorem \ref{thm:hom_linear}, and define the quantized feedback as:
		 \begin{equation}\label{eq:u_q}
		 	u(\mathfrak{q}(x))=K_0\mathfrak{q}(x)+ \|\mathfrak{q}(x)\|_\dn^{1+\mu}K\dn(-\ln\|\mathfrak{q}(x)\|_\dn)\mathfrak{q}(x)
		 \end{equation}
	 where $\mathfrak{q}:\mathbb{R}^n\mapsto \mathcal{Q}$ is a quantization function.   If $K_0=\mathbf{0}$ and the quantization error satisfies	
	\begin{equation}\label{eq:con1}
		\|\mathfrak{q}(x)-x\|_\dn\le \epsilon \|x\|_\dn, \quad \forall x\in\mathbb{R}^n, \ \epsilon\in\mathbb{R}_+,
	\end{equation}
}	\textit{then, for a sufficiently small $\epsilon$, the system \eqref{eq:q_sys}, \eqref{eq:u_q} is %{\color{blue}globally uniformly asymptotically stable, and}
}
	\begin{itemize}
		\item \textit{globally uniformly finite-time stable for $\mu<0$;}
%		\item \textit{globally uniformly exponentially stable for $\mu=0$;}
		\item \textit{globally uniformly nearly fixed-time stable for $\mu>0$.}
	\end{itemize}
	\textit{
	If $K_0\neq \mathbf{0}$ and the quantization error additionally satisfies
	\begin{equation}\label{eq:con2}
		\|\mathfrak{q}(x)-x\|\le\kappa\|x\|, \quad \forall x\in\mathbb{R}^n, \ \kappa\in\mathbb{R}_+,
	\end{equation}
	 then the system \eqref{eq:q_sys}, \eqref{eq:u_q} is %{\color{blue}globally uniformly practically asymptotically stable for $\mu\ge 0$, and}
}
	\begin{itemize}
		\item \textit{locally uniformly finite-time stable for $\mu<\min\{\underline{\eta}-1, 0\}$;}
		%\item \textit{globally uniformly exponentially stable for $\mu=0$;}
		\item \textit{globally uniformly practically fixed-time stable \footnote{
  The system \eqref{eq:lin_sys} is globally  uniformly practically
  \begin{itemize}
      \item Lyapunov stable if  
 $\exists r\!\in\! \mathbb{R}_+, \exists \chi\in\mathcal{K}$ : $\|x(t)\|\!\le\! r + \chi(\|x_0\|)$, $\forall t\!\ge\! t_0$;
  \item fixed-time stable if it is globally uniformly practically Lyapunov stable and $\exists \tilde T>0$ such that $\|x(t)\| \leq r, \;\; \forall t\geq \tilde T,  \ \forall x_0\in \R^n.$
  \end{itemize}
   } for $\mu\!>\!\max\{\overline{\eta}\!-\!1, 0\}$}. 
	\end{itemize}

\end{theorem}

\begin{proof}
	The inequality \eqref{eq:con1} implies that   $\mathfrak{q}(\boldsymbol{0})=\boldsymbol{0}$ and $x=\boldsymbol{0}$ is the equilibrium of  the system.
Since $\|x\|_\dn$ serves as a Lyapunov function for the quantization-free system \eqref{eq:lin_sys}, let us calculate the derivative along with the system using quantization \eqref{eq:q_sys}:
\begin{equation}
	\begin{aligned}
		&\tfrac{d}{dt}\|x\|_\dn = \|x\|_\dn\tfrac{x^\top\dn(-s_{x})P\dn(-s_{x})}{x^\top\dn(-s_{x})PG_\dn \dn(-s_{x})x} \left(Ax+ Bu(\mathfrak{q}(x))\right)\\
		&\!=\!\|x\|_\dn\tfrac{x^\top\dn(-s_{x})P\dn(-s_{x})}{x^\top\dn(-s_{x})PG_\dn \dn(-s_{x})x} \!\left(Ax\!+\! Bu(x) \!+\! Bu(\mathfrak{q}_x)\!-\!Bu(x)\right)
	\end{aligned}
\end{equation}
where $s_{x}=\ln\|x\|_\dn$ and $\mathfrak{q}_x=\mathfrak{q}(x)$.  The equation \eqref{eq:G0} leads to $A_0\dn(s)=e^{\mu s}\dn(s)A_0$, $\dn(s)B = e^{s}B, \forall s\in \R$ and
\begin{equation}
	\tfrac{d}{dt}\|x\|_\dn = -\rho\|x\|_\dn^{1+\mu}+\tfrac{x^\top\dn(-s_{x})PB\left(u(\mathfrak{q}_x)-u(x)\right)}{x^\top\dn(-s_{x})PG_\dn \dn(-s_{x})x}. 
\end{equation}
The control $u(x)$ consists of two parts: the linear part $K_0x$ and the homogeneous part $\tilde{u}(x) = \|x\|_{\dn}^{1+\mu}K\dn(-\ln\|x\|_{\dn})x$. Therefore, we derive
\begin{equation}
	\begin{aligned}
		\tfrac{d}{dt}\|x\|_\dn \!=\! -\rho\|x\|_\dn^{1+\mu}&+\tfrac{x^\top\dn(-s_{x})PBK_0(\mathfrak{q}_x-x)}{x^\top\dn(-s_{x})PG_\dn \dn(-s_{x})x}\\
		& + \tfrac{x^\top\dn(-s_{x})PB\left(\tilde{u}(\mathfrak{q}_x)-\tilde{u}(x)\right) }{x^\top\dn(-s_{x})PG_\dn \dn(-s_{x})x}.
	\end{aligned}
\end{equation}
Referring the definition of canonical homogeneous norm $\|\dn(-s_{x})x\|=1$, then using Cauchy–Schwarz inequality, it yields that
\begin{equation}
	\tfrac{d}{dt}\|x\|_\dn \!\le\! -\!\rho\|x\|_\dn^{1+\mu}\!+\!c_1\|\mathfrak{q}_x\!-\!x\| \!+\! \tfrac{x^\top\dn(-s_{x})PB\left(\tilde{u}(\mathfrak{q}_x)\!-\!\tilde{u}(x)\right) }{x^\top\dn(-s_{x})PG_\dn \dn(-s_{x})x},
\end{equation}
where $c_1 = \tfrac{\sqrt{\lambda_{\max}(K_0^\top B^\top PBK_0)}}{\lambda_{\min}(P^{1/2}G_\dn P^{-1/2} + P^{-1/2}G_\dn^\top P^{1/2})}$.

Since $B\tilde{u}(\dn(s)z) = e^{(1+\mu)s}B\tilde{u}(z), \forall s\in \R, \forall z\in \R^n$ then 
\begin{equation}\label{eq:d_norm_1}
	\begin{aligned}
		\tfrac{d}{dt}\|x\|_\dn \!\le\!& -\rho\|x\|_\dn^{1+\mu}+c_1\|\mathfrak{q}_x-x\|\\
		& + c_2\|x\|_\dn^{1+\mu}\|\tilde{u}\left(\dn(-s_x)\mathfrak{q}_x\right)\!-\!\tilde{u}(\dn(-s_{x})x)\|,
	\end{aligned}
\end{equation}
where $c_2 = \tfrac{\sqrt{\lambda_{\max}(K^\top B^\top PBK)}}{\lambda_{\min}(P^{1/2}G_\dn P^{-1/2} + P^{-1/2}G_\dn^\top P^{1/2})}$.

Taking into account the monotonicity of the dilation $\dn$ (see Definition \eqref{def:monocity}), the condition \eqref{eq:con1} implies that 
\begin{equation}\label{eq:con_d}
	\begin{aligned}
 &\quad \;\;e^{-\ln\epsilon}e^{-\ln\|x\|_\dn}\|\mathfrak{q}_x-x\|_{\dn}\le 1\\
 &\Rightarrow\|\dn(-\ln\epsilon)\dn(-\ln\|x\|_\dn)(\mathfrak{q}_x-x)\|_{\dn}\le 1\\
	    &\Rightarrow \|\dn(-\ln\epsilon)\dn(-\ln\|x\|_\dn)(\mathfrak{q}_x-x)\|\le 1\\
     &\Rightarrow \lfloor\dn(-\ln\epsilon)\rfloor\cdot\|\dn(-\ln\|x\|_\dn)(\mathfrak{q}_x-x)\|\le 1 
	\end{aligned}
\end{equation}
Due to \eqref{eq:d_bound}, there exists $\alpha\in\mathcal{K}$ such that $\tfrac{1}{\lfloor\dn(-\ln\epsilon)\rfloor}\le \alpha(\epsilon)$ leading to
  \begin{align*}
  	\|\dn(-\ln\|x\|_\dn)\mathfrak{q}_x\|\! &=\! \|\dn(-\ln\|x\|_\dn)x \!+\! \dn(-\ln\|x\|_\dn)(\mathfrak{q}_x\!-\!x)\|\\
  	&\le 1+ \alpha(\epsilon)
  \end{align*}
for all $x\neq 0$ with a sufficiently small $\epsilon>0$.

Since $\dn(-\ln\|x\|_\dn)x$ lies on the unit sphere for any $x\neq 0$, for a sufficiently small $\epsilon>0$ 
there exists a compact set, which does not contain the origin, such that $\dn(-\ln\|x\|_\dn)\mathfrak{q}_x$ and $\dn(-\ln \|x\|_{\dn})x$ always belong to this compact set for all $x\neq \mathbf{0}$. Since   $\tilde{u}$ is continuous on $\mathbb{R}^n\backslash\{\mathbf{0}\}$, then it is uniformly continuous on the compact set. Therefore, there exists a class-$\mathcal{K}$ function $\gamma$, such that
\begin{equation}\label{eq:class_K}
		\!\left\|   \tilde{u}(\dn(-s_x)\mathfrak{q}_x)\!-\!\tilde{u}(\dn(-s_{x})x)\right\|
		\!\!\le\! \gamma\!\left(\|\dn(-s_{x})(\mathfrak{q}_x\!-\!x)\| \right)\!, \forall x\!\neq\! \mathbf{0}
\end{equation}
Then, using the latter inequality and \eqref{eq:con_d}, we derive  
\begin{equation}\label{eq:d_norm_2}
	\begin{aligned}
		\tfrac{d}{dt}\|x\|_\dn \!\le\! -(\rho-c_2\overline{\gamma}(\epsilon))\|x\|_\dn^{1+\mu}+c_1\|\mathfrak{q}_x-x\|,
	\end{aligned}
\end{equation}
where $\overline{\gamma}=\gamma\circ\alpha \in\mathcal{K}$.

For  $K_0=\boldsymbol{0}\Rightarrow c_1 = 0$ we derive 
 \begin{equation}
 	\begin{aligned}
 		\tfrac{d}{dt}\|x\|_\dn \!\le\! -(\rho-c_2\overline{\gamma}(\epsilon))\|x\|_\dn^{1+\mu}.
 	\end{aligned}
 \end{equation}
and if  $0<\epsilon< \overline{\gamma}^{-1}\left(\tfrac{\rho}{c_2}\right)$, where $\overline{\gamma}^{-1}$ is an inverse function of $\overline{\gamma}$, then the system is globally finite-time stable for $\mu<0$, exponentially stable for $\mu=0$ and nearly fixed-time stable for $\mu>0$.

If $K_0\neq \mathbf{0}$  (i.e., $c_1\neq 0$), let us consider the relation between the Euclidean norm and the canonical homogeneous norm given in Proposition \ref{prop:3}. It yields that 
\begin{equation}
	 \left\{
	\begin{aligned}
		&\|\mathfrak{q}\!-\!x\| \!\le\! \kappa\|x\|_\dn^{\underline{\eta}}, \ \|x\|_\dn\!\le\! 1,\\
		&\|\mathfrak{q}\!-\!x\|\!\le\! \kappa\|x\|_\dn^{\overline{\eta}}, \  \|x\|_\dn\!\ge\! 1,
	\end{aligned}
	\right.
\end{equation}

Using the latter estimate and \eqref{eq:d_norm_2}, we derive that
\begin{equation}\label{eq:d_norm_3}
	\tfrac{d\|x\|_\dn}{dt} \!\le\!\!\left\{\begin{aligned}
		 &\!\!-\!\!\left[\rho\!-\!c_2\overline{\gamma}(\epsilon) \!-\! c_1\kappa\|x\|_\dn^{\underline{\eta}-(1+\mu)}\right]\|x\|_\dn^{1+\mu},\ \|x\|_\dn\!\le\!1,\\
		 &\!\!-\!\!\left[\rho\!-\!c_2\overline{\gamma}(\epsilon) \!-\! c_1\kappa\|x\|_\dn^{\overline{\eta}-(1+\mu)}\right]\|x\|_\dn^{1+\mu},\ \|x\|_\dn\!\ge\!1.
	\end{aligned}
\right.
\end{equation}
for some $\kappa>0$. Let  $0<\epsilon < \overline{\gamma}^{-1}(\frac{\rho}{c_2})$, and we have the following two cases:
\begin{enumerate}
    \item 
If $\mu > \max\{\overline{\eta}-1,\underline{\eta}-1\}=\overline{\eta}-1$
we have $\|x\|_\dn^{\overline{\eta}-(1+\mu)} \to 0$ and $\|x\|_\dn^{\underline{\eta}-(1+\mu)} \to 0$ as $\|x\|_\dn \to \infty$. Thus, the system is globally practically fixed-time stable due to $\mu>0$.

\item If $\mu \le \min\{\overline{\eta}-1, \underline{\eta}-1\}=\underline{\eta}-1$
then  $\|x\|_\dn^{\overline{\eta}-(1+\mu)} \to 0$ and $\|x\|_\dn^{\underline{\eta}-(1+\mu)} \to 0$ as $\|x\|_\dn \to 0$. Hence, the system \eqref{eq:q_sys} is locally finite-time stable if $\mu<0$.
\end{enumerate}

%If $\mu = 0$ then  $\|x\|=\|x\|_{\dn}$ and $\overline{\eta}=\underline{\eta}=1$, i.e. $\|x\|_\dn^{\overline{\eta}-(1+\mu)} =\|x\|_\dn^{\underline{\eta}-(1+\mu)}= 1$. Thus,  the system \eqref{eq:q_sys} is globally exponentially stable.
The proof is complete.
$\hfill \square$
\end{proof}

For $K_0\neq 0$, the region of attraction (resp., the attractive set) of locally finite-time (resp., globally practically fixed-time) stable system \eqref{eq:q_sys}, \eqref{eq:u_q} with $\mu<0$ (resp. $\mu>0)$ can be tuned arbitrarily large (resp., small) by a selection of a small enough 
$0<\epsilon < \overline{\gamma}^{-1}(\frac{\rho}{c_2})$ and a small enough  $\kappa>0$.

\begin{remark}
	\textit{
	The condition  $\|\mathfrak{q}(x) - x\|\le\epsilon\|x\|$ is sufficient for exponential stabilization of a linear control system  with states quantization (see, e.g., \cite{Fu_etal_2005_TAC}, \cite{bullo_Liberzon2006TAC}). Similarly, for $\dn$-homogeneous LTI system with the $\dn$-homogeneous quantized feedback, the finite/nearly fixed-time stabilization condition is expressed in the same form but utilizing the canonical homogeneous norm \eqref{eq:con1}. Furthermore, for $G_\dn = I_n$ (i.e., $\mu=0$)  we have $\|x\|_{\dn}=\|x\|$. This corresponds to the linear case considered above.} 
\end{remark}

At first sight, the restriction to a quantizer in Theorem \ref{thm:main} is not easy to check since it is expressed in terms of the canonical homogeneous norm defined implicitly. However, in the next section, we show that using the well-known logarithmic quantizer, the required condition can be satisfied.

\subsection{Finite/fixed-time stabilization under logarithmic quantization}
The objective of this section is to study the feasibility of finite/fixed-time stabilization problem with state quantization using  logarithmic quantizer (see Fig. \ref{fig:quantizer}). Let us briefly recall the corresponding definitions.
The distinguishing property of a logarithmic quantizer is that the quantization error decreases as the states approach the origin. This allows the convergence/stability property of the original system to be preserved.

\textbf{\textit{Logarithmic quantizer}} is described by
\begin{equation}\label{eq:log_quantizer}
	q(\phi)\!=\!\left\{ 
	\begin{aligned}	
	&\nu^i \zeta_0, \  \tfrac{1}{1+\delta} \nu^i \zeta_0 \!<\! \phi \!\leq \tfrac{1}{1-\delta} \nu^i \zeta_0,  i\!=\!0, \pm 1, \pm 2, \ldots \\ 
	&	0, \  \qquad   \phi = 0\\ 
	& -q(-\phi), \  \phi<0
	\end{aligned}\right.
\end{equation}
where $q(\phi), \phi\in\mathbb{R}$, $\zeta_0\in\mathbb{R}_+$, $\nu \in(0,1)$ represents the quantization density and $\delta=(1-\nu) /(1+\nu)$.
A small $\nu$ (resp., large $\delta$) implies a coarse quantization, but a large $\nu$ (resp., a small $\delta$) means a dense quantization. 

For the logarithmic quantizer, the quantization error is sector bounded \cite{Fu_etal_2005_TAC}:
\begin{equation}\label{eq:log_q}
	|q(\phi) - \phi|\le \delta|\phi|, \ \delta \in (0,1)
\end{equation} 
and vanishing as the state goes to the origin.
\begin{figure}[ht]
	\centering
\includegraphics[width=0.4\textwidth]{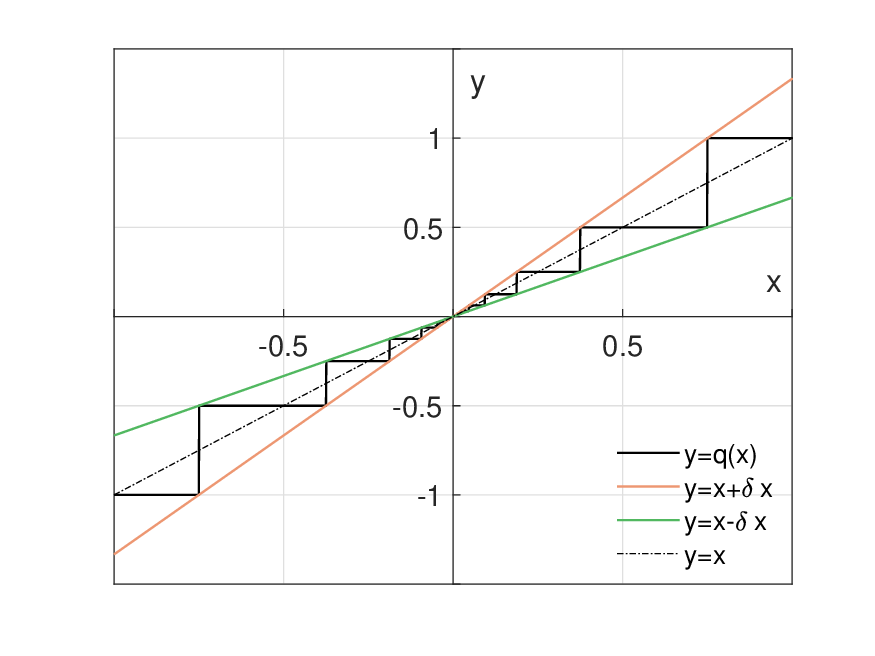}
	\caption{  Logarithmic quantizer $q(x)$.}
	\label{fig:quantizer}
\end{figure}
As shown in \eqref{eq:log_q}, the quantization error of the logarithmic quantizer is formulated as an absolute value. To validate the stability condition \eqref{eq:con1}, a new relation between $|\cdot|$ and $\|\cdot\|_\dn$ has been derived in Lemma \ref{lem:1}.
Based on the results of Lemma \ref{lem:1}, we construct a logarithmic quantization function $\mathfrak{q}(x)$ that employs the transformation matrix $\Gamma$ and fulfills the conditions of Theorem \ref{thm:main}.

\begin{corollary}\label{cor:2}
	\textit{Let $G_{\dn}$ be a diagonizable generator of the linear dilation in $\R^n$:
 \[
 \exists \Gamma\!\in\! \R^{n\times n} : \;\;G_{\dn}\!=\!\Gamma \Lambda \Gamma^{-1}, \;\; 
 \Lambda\!=\!\mathrm{diag}(\lambda_1,\lambda_2,...,\lambda_n)\!\succ\! 0
 \]
 and the canonical homogeneous norm $\|x\|_{\dn}$ be induced by the weighted Euclidean norm $\|x\|=\sqrt{x^{\top}Px}$, where $x\in \R^n$ and $0\prec P=P^{\top}\in \R^{n\times n}$ satisfies \eqref{eq:mon_cond}. If  $q$ is a logarithmic quantization function with parameter $\delta>0$ and}
 \begin{equation}\label{eq:log_quan_Gamma}
     \mathfrak{q}(x) = \Gamma[q(\xi_{x,1}), q(\xi_{x,2}),\cdots, q(\xi_{x,n})]^\top,
 \end{equation}
 \textit{
 where $\xi_{x} = \Gamma^{-1}x = [\xi_{x,1}, \xi_{x,2},\cdots, \xi_{x,n}]^\top$,  $x\in\mathbb{R}^n$,
then for any $\epsilon>0$ and any $\kappa>0$ there exists $\delta>0$ such that the inequalities \eqref{eq:con1}, \eqref{eq:con2} are fulfilled simultaneously.
}
\end{corollary}
\begin{proof} On the one hand,  from the formula of quantization function $\mathfrak{q}$, we have
{\small$$
\Gamma^{-1}(\mathfrak{q}(x)-x) \!=\! [q(\xi_{x,1})-\xi_{x,1}, q(\xi_{x,2})-\xi_{x,2},\cdots, q(\xi_{x,n})-\xi_{x,n}]^\top,
$$} where
$\Gamma^{-1}x = [\xi_{x,1}, \xi_{x,2},\cdots, \xi_{x,n}]^\top$.
Then, since $G_\dn$ is diagonalizable, according to  the quantization error of the logarithmic quantizer \eqref{eq:log_q} and Lemma \ref{lem:1}, we conclude  that for any given $\epsilon>0$, there exists a sufficiently small $\delta>0$ such that the inequality $|q(\xi_{x,i})-\xi_{x,i}|\le \delta |\xi_{x,i}|$, $\forall i=1,2,\cdots, n$ implies the inequality \eqref{eq:con1}.

On the other hand, since $\Gamma$ is invertible, then  $\Gamma^{-\top}\Gamma^{-1}$ is positive definite, hence the inequality \eqref{eq:log_q} also implies  that 
\begin{equation}\label{eq:}
    \begin{aligned}
        \|\mathfrak{q}(x) - x\|_2& = \|\Gamma\left([q(\xi_{x,1}), q(\xi_{x,2}),\cdots, q(\xi_{x,n})]^\top - \xi_x\right)\|_2\\
        &\le
        \lambda_{\max}^{1/2}(\Gamma^\top \Gamma) \left(\sum_{i=1}^{n}|q(\xi_{x,i})-\xi_{x,i}|^2\right)^{1/2}\\
        &\le
        \lambda_{\max}^{1/2}(\Gamma^\top \Gamma) \left(\sum_{i=1}^{n}|\delta\xi_{x,i}|^2\right)^{1/2}\\
        &\le
        \delta\lambda_{\max}^{1/2}(\Gamma^\top \Gamma) \lambda_{\max}^{1/2}(\Gamma^{-\top} \Gamma^{-1})\|x\|_2
    \end{aligned}
\end{equation}
The latter yields \eqref{eq:con2} with 
\begin{equation}\label{eq:kappa}
\kappa=\delta\lambda_{\max}^{1/2}(\Gamma^\top \Gamma) \lambda_{\max}^{1/2}(\Gamma^{-\top} \Gamma^{-1})\tfrac{\sqrt{\lambda_{\min}(P)}}{\sqrt{\lambda_{\max}(P)}}.
\end{equation}
The proof is complete. 
$\hfill \square$
\end{proof}

The proven corollary immediately implies that all conditions of Theorem \ref{thm:main} are fulfilled for the constructed quantizer \eqref{eq:log_quan_Gamma} with a sufficiently small $\delta>0$ provided that the linear dilation $\dn$ has a diagonalizable generator. Therefore, the corresponding quantized control \eqref{eq:u_q} is a global
(local) homogeneous stabilizer of the system \eqref{eq:lin_sys} if $K_0=\mathbf{0}$
(resp., $K_0\neq \mathbf{0}$).
\begin{remark}
\textit{For the controlled chain of integrators: 
$$A=\left(\begin{smallmatrix} 0 & 1 & 0 & ...  & 0 & 0\\ 0 & 0 &1 & ...& 0 & 0\\
... & ... & ... & ... & ... & ...\\
0 & 0 & 0 & ... & 0 & 1\\
0 & 0 & 0 & ... & 0 & 0\\
\end{smallmatrix}\right), \quad B=\left(\begin{smallmatrix} 0  \\ 0\\
... \\
0 \\
1\\
\end{smallmatrix}\right)
$$
the dilation $\dn$ has a diagonal generator, so $\Gamma=I_n$ and all conditions of Theorem \ref{thm:main} hold even for the classical logarithmic quantizer: $\mathfrak{q}(x)\!=\![q(x_1),...,q(x_n)]^{\top}$, $x\!=\!(x_1,...,x_n)^{\top}\!\in\! \R^n$.
}
\end{remark}
%}

%{\color{blue}
%As demonstrated in latter corollary, the selection of different homogeneity degrees enables global practical fixed-time stability or local finite-time stability. Consequently, the global finite/fixed stability of the origin can be ensured by utilizing a commutation strategy. The commutation of two continuous-time homogeneous controllers has been well developed when the switch signal is determined by the state (see, e.g. \cite{polyakov_etal_2015_Aut}, \cite{polyakov2016IJRNC}). However, switching control poses a significant challenge in scenarios where only quantization is transmitted.}
%The following corollary proves that switching the homogeneity degree any controlled LTI system can be globally stabilized to the origin in a finite/fixed time by means of a locally homogeneous controller with logarithmic quantization of transformed state measurements. 
The global fixed-time stabilizer can be designed by a combination of various homogeneous stabilizers with negative and positive degrees.
\begin{corollary}\label{cor:comu}
	\textit{
 Let $G_0$, $Y_0$, and $K_{0}=Y_0(G_0-I_n)^{-1}$, $A_0=A+BK_0$ be defined as in Theorem \ref{thm:hom_linear}. Let $G_0$ be diagonalizable
 \[
 \exists \Gamma\!\in\! \R^{n\times n} : \;\;G_0\!=\!\Gamma \Lambda_0 \Gamma^{-1}, \;\; 
 \Lambda_0\!=\!\mathrm{diag}(\tilde \lambda_1,\tilde \lambda_2,...,\tilde \lambda_n)\!\succ\! 0
 \]
  and the logarithmic quantizer $\mathfrak{q}$ be defined by the formula \eqref{eq:log_quan_Gamma} with a parameter $\delta>0$.	  
	 Suppose that the feedback control with quantized states is defined as follows:
	\begin{equation}\label{eq:u_c}
		u \!=\!\!    \left\{
		\begin{smallmatrix}
			\!\!\!K_0\mathfrak{q}(x) + \|\mathfrak{q}(x)\|_{\dn_1}^{1+\mu_1}\!K\dn_1(\!-\!\ln\|\mathfrak{q}(x)\|_{\dn_1}\!)\mathfrak{q}(x) &\text{ if } & \|\mathfrak{q}(x)\|\ge 1, \\
			\!\!\!K_0\mathfrak{q}(x) + \|\mathfrak{q}(x)\|_{\dn_2}^{1+\mu_2}\!K\dn_2(\!-\!\ln\|\mathfrak{q}(x)\|_{\dn_2}\!)\mathfrak{q}(x) & \text{ if } & \|\mathfrak{q}(x)\|< 1,
		\end{smallmatrix}
		\right.
	\end{equation}
where  $K= YX^{-1}$, $\dn_i(\cdot)$, $i=1,2$ are generated by $G_{\dn_i}=G_0+\mu_iI_n$, with $0\leq \mu_1\leq 1/n$ and $-1\leq \mu_2<0$,  $\|\cdot\|_{\dn_i}$ are induced by the norm $\|x\|\!=\!$ $\sqrt{x^{\top} P x}$, $P=X^{-1}$.
If  $X \in \mathbb{R}^{n \times n}$ and $Y \in \mathbb{R}^{m \times n}$ satisfy the following algebraic system: 
\begin{equation}\label{eq:LMI1}\small
			\left\{\!\begin{aligned}
			&\!X A_{0}^{\top}\!+\!A_{0} X\!+\!Y^{\top} B^{\top}\!+\!B Y\prec\!0, \quad   X \succ 0,\\
			&\!2(1+\mu_1)X\succeq X G_{\dn_1}^{\top}+G_{\dn_1} X \succ 0,\\
   & XG_{\dn_2}^{\top}+G_{\dn_2} X \succeq 2(1+\mu_2)X,
		\end{aligned}\right.
	\end{equation}
 then for sufficiently small $0\!<\!\delta\!<\!1$,  the system \eqref{eq:lin_sys}, \eqref{eq:u_c} is: 
}
	\begin{itemize}
		\item \textit{globally uniformly finite-time stable for $\mu_1=0$;}
		\item \textit{globally uniformly fixed-time stable for $\mu_1>0$;}
	\end{itemize} 
\end{corollary}
\begin{proof}
Let us consider the following Lyapunov function candidate:
\begin{equation}
    V = \left\{
\begin{aligned}
    &\|x\|_{\dn_1}, & \|x\|\ge 1\\
    &\|x\|_{\dn_2}, & \|x\|\le 1\\
\end{aligned}
    \right.
\end{equation}
Notice that $V\in C^{\R^n}\cap C^{1}(\R^{n}\backslash S\backslash\{\mathbf{0}\})$, where $S=\{x\in \R^n: \|x\|=1\}$.
For sufficiently small $\delta\in (0,1)$ we have 
	\begin{equation}
		(1-\kappa)\|x\| \leq \|\mathfrak{q}(x)\| \leq (1+\kappa)\|x\|,
	\end{equation}
  where $\kappa=\kappa(\delta)$ has the form as \eqref{eq:kappa}, then 
 \begin{equation}
		\tfrac{1}{(1-\kappa)}\!\leq\! \|x\| \Rightarrow 1 \!\leq\! \|\mathfrak{q}(x)\|,\!\; \|x\|\!\le\! \tfrac{1}{(1+\kappa)}\Rightarrow \|\mathfrak{q}(x)\|\!\le\! 1.
	\end{equation}
 
 Let $\Omega_1:=\{x\in\mathbb{R}^n:\tfrac{1}{(1-\kappa)}\leq \|x\|\}$, $\Omega_2:=\{x\in\mathbb{R}^n:\|x\|\le \tfrac{1}{(1+\kappa)}\}$, $\Omega_3:=\{x\in\mathbb{R}^n:\tfrac{1}{(1+\kappa)}\leq \|x\|\le \tfrac{1}{(1-\kappa)}\}$.
 {By Corollary \ref{cor:2}, the logarithmic quantizer  $\mathfrak{q}(x)$ satisfies the conditions \eqref{eq:con1} and \eqref{eq:con2} for a sufficiently small $0 < \delta < 1$.} Thus
 according to Theorem \ref{thm:main} and  \eqref{eq:d_norm_3}, the derivative of $V$ admits the estimate {\small\begin{equation}\label{eq:d_norm_4}
	\dot{V} \!\le\!\!\left\{\begin{aligned}
		 &\!\!-\!\!\left[\rho_1\!-\!c_{2,\mu_1}\overline{\gamma}_1(\epsilon) \!-\! c_{1,\mu_1}\kappa\|x\|_{\dn_1}^{\overline{\eta}_1-(1+\mu_1)}\right]\|x\|_{\dn_1}^{1+\mu_1},\ x\!\in\!\Omega_1\\
   &\!\!-\!\!\left[\rho_2\!-\!c_{2,\mu_2}\overline{\gamma}_2(\epsilon) \!-\! c_{1,\mu_2}\kappa\|x\|_{\dn_2}^{\underline{\eta}_2-(1+\mu_2)}\right]\|x\|_{\dn_2}^{1+\mu_2},\ x\!\in\!\Omega_2,
	\end{aligned}
\right.
\end{equation}
}
{where $c_{1,\mu_i} \!=\! \tfrac{\lambda_{\max}^{1/2}(K_0^\top B^\top PBK_0)}{h_i}$, $c_{2,\mu_i}\!=\!\tfrac{\lambda_{\max}^{1/2}(K^\top B^\top PBK)}{h_i}$,\\ $\rho_i = \tfrac{\lambda_{max}(X A_{0}^{\top}\!+\!A_{0} X\!+\!Y^{\top} B^{\top}\!+\!B Y)}{h_i}$, \\
$h_i=\lambda_{\min}(P^{1/2}G_{\dn_i} P^{-1/2} \!+\! P^{-1/2}G_{\dn_i}^\top P^{1/2})$, $i=1,2$.
The functions $\overline{\gamma}_1, \overline{\gamma}_2\in\mathcal{K}$ are obtained from \eqref{eq:con_d}, \eqref{eq:class_K} with $(\dn_i,\mu_i)$, $i=1,2$.
}

The LMIs defined in (\ref{eq:LMI1}) guarantee $\overline{\eta}_1-1\le \mu_1$ and $\mu_2\le\underline{\eta}_2-1$, where $\overline{\eta}_i\!=\!\tfrac{1}{2} \lambda_{\max }(X^{-\frac{1}{2}} G_{\mathrm{d}_i} X^{\frac{1}{2}}+X^{\frac{1}{2}} G_{\mathrm{d}_i}^{\top} X^{-\frac{1}{2}})$, $
\underline{\eta}_i=\tfrac{1}{2} \lambda_{\min }(X^{-\frac{1}{2}} G_{\mathrm{d}_i} X^{\frac{1}{2}}+P^{-\frac{1}{2}} G_{\mathrm{d}_i}^{\top} X^{-\frac{1}{2}})$, $i=1,2$.
{
Then, we have $\|x\|_{\dn_1}^{\overline{\eta}_1-(1+\mu_1)}$ monotonically decreasing to zero as $\|x\|_{\dn_1}$ increases to infinity, and $\|x\|_{\dn_2}^{\overline{\eta}_2-(1+\mu_2)}$ monotonically decreasing to zero as $\|x\|_{\dn_2}$ decreases to zero. Thus, we conclude that for a sufficiently small $\delta$ we have
\[
\dot{V}<0, \ x\in\Omega_1\cup\Omega_2.
\]
}
On the other hand, for $\tfrac{1}{1+\kappa}\le\|x\|\le \tfrac{1}{1-\kappa}$, the quantization error has 
$$
\|\mathfrak{q}_x-x\|\le \tfrac{\kappa}{1-\kappa}.
$$
Using  \eqref{eq:d_norm_2}, for a sufficiently small $\delta>0$
we derive 
\begin{equation}\label{eq:d_norm_5}
	\begin{aligned}
		\dot{V} \!\le\! \max\limits_{i=1,2}\left\{\!-\!\tfrac{\rho_i\!-\!c_{2,\mu_i}\overline{\gamma}_i(\epsilon)}{(1+\kappa)^{1+\mu_i}}\!+\!c_{1,\mu_i}\tfrac{\kappa}{1-\kappa}\right\}<0,\ x\!\in\!\Omega_3
	\end{aligned}
\end{equation}
	Therefore, $V$ is a global Lyapunov function and the system is globally asymptotically stable. Since $\mu_2<0$, then by Theorem \ref{thm:main} it is finite-time stable if $\mu_1=0$
		and fixed-time stable if 
		$\mu_1>0$. The proof is complete.
$\hfill \square$
\end{proof}

%\Andrey{End of the corrected part. Please make correction, next, I will read it again.}
%For the system \eqref{eq:lin_sys}, \eqref{eq:hom_con_P}, the canonical homogeneous norm serves as a Lyapunov function. The canonical homogeneous norm $\|\cdot\|_\dn$ is not a norm in the conventional sense since it does not satisfy the triangle inequality. Nevertheless, it defines a norm topology in $\mathbb{R}^n$.
%The Lemma \ref{lem:1} establishes a connection between the Euclidean and canonical homogeneous norms, offering a valuable tool for stabilizing multiple-input systems with scalar quantizers. Moreover, the Corollary \ref{cor:2} demonstrates that despite the logarithmic levels exhibiting exponential growth as the states increase, it remains feasible to maintain the finite/nearly fixed-time stability of the homogeneous quantization stabilization system.
%In the next section, a numerical validation is conducted.

\section{Numerical simulation}\label{section:sim}

{
To demostrate the finite/fixed-time stabilization via the numerical simulations, we consider two LTI systems: 
\begin{equation*}
       A_1 \!=\! \left[\begin{smallmatrix}
        0 & 1 & 0\\
        0 & 0 & 1\\
     0  & 0 & 0
\end{smallmatrix}\right],\  B_1 \!=\! \left[\begin{smallmatrix}
        0 & 0\\
        0 & 0\\
        1 & 2
\end{smallmatrix}\right],\
 A_2 \!=\! \left[\begin{smallmatrix}
        1 & 1 & 0\\
        0 & 0 & -0.5\\
     0.2  & 0 & 0.5
\end{smallmatrix}\right],\  B_2 \!=\! \left[\begin{smallmatrix}
        0 & 0\\
        0 & 1\\
        1 & 2
    \end{smallmatrix}\right].
\end{equation*}
The matrix $A_1$ is nilpotent and $(A_i,B_i)$, $i=1,2$ are controllable, with initial condition $x_0=[20,20,10]^\top$.
For system \eqref{eq:lin_sys}, \eqref{eq:u_q} with $A=A_1$, $B=B_1$, $K_0 = \boldsymbol{0}$, the control parameters are obtained by solving the LMIs \eqref{eq:LMI0}. For negative degree $\mu=-0.2$, $\rho = 2$, we have 
$$
K \!=\! \!-\!\left[\begin{smallmatrix}
    5.65 &  4.51 &  1.44\\
  11.30  & 9.03 &  2.88
\end{smallmatrix}\right], G_\dn \!=\!\! \left[\begin{smallmatrix}
    1.4 &  0 &        0\\
    0  &  1.2 &        0\\
    0  & 0  &  1
\end{smallmatrix}\right],  P \!=\!\! \left[\begin{smallmatrix}
    70.88 &  39.18  &  7.02\\
   39.18 &  22.32  &  4.11\\
    7.02  &  4.11  &  0.79
\end{smallmatrix}\right]\!.
$$
For positive degree $\mu=0.2$, $\rho = 2$ the parameters are $$K \!=\!\! -\left[\begin{smallmatrix}
    0.87  & 1.68  & 0.96\\
   1.73 &  3.23 &  1.92
\end{smallmatrix}\right], 
G_\dn \!=\!\! \left[\begin{smallmatrix}
    0.6 &  0 &        0\\
    0  &  0.8 &        0\\
    0  & 0  &  1
\end{smallmatrix}\right],  P \!=\!\! \left[\begin{smallmatrix}
    7.61  &  7.08  &  2.11\\
    7.08  &  7.67  &  2.39\\
    2.11  &  2.39 &   0.86
\end{smallmatrix}\right].
$$
The simulation results of the closed-loop system
\eqref{eq:lin_sys}, \eqref{eq:u_q} with $A=A_1,B=B_1$ are presented in 
Fig. \ref{fig:2}. The obtained trajectories are typical for finite-time and nearly fixed-time convergence. The simulation was conducted with the parameter   $\nu=0.9$ of the logarithmic quantizer (see, the formula \eqref{eq:log_quantizer}). 

For a non-nilpotent system \eqref{eq:lin_sys} with $A=A_2$, $B=B_2$, we design a global stabilizer \eqref{eq:u_c} by Corollary \ref{cor:comu}. The control parameters are obtained by solving LMIs \eqref{eq:LMI1} with $\mu_1 = 0.2$, $\mu_2=-0.5$ for fixed-time stabilization and $\mu_1 = 0$, $\mu_2=-0.5$ for finite-time stabilization:
$$ K = \left[\begin{smallmatrix}
    10.09  &  2.66 &  -0.61\\
   -5.05 &  -1.33  & -0.0001
\end{smallmatrix}\right],P \!=\!\! \left[\begin{smallmatrix}
    0.4959  &  0.1629  &  0\\
    0.1629  &  0.1179  & 0\\
    0  & 0  &  0.1131
\end{smallmatrix}\right]$$
$$K_0 \!=\! \left[\begin{smallmatrix}
    1.8  &  2 & -1.5\\
   -1 &  -1 &    0.5\\
\end{smallmatrix}\right], G_{\dn_1} \!=\!\! \left[\begin{smallmatrix}
    0.8  &  0 &  0\\
    0.2  &  1  &  0\\
   -0.0003  &  0  &  1
\end{smallmatrix}\right], G_{\dn_2} \!=\!\! \left[\begin{smallmatrix}
    1.5  &  0 &  0\\
    -0.5  &  1  &  0\\
   0.0007  &  0  &  1
\end{smallmatrix}\right].$$
The simulation results for this case  are presented in Fig.  \ref{fig:3}. They also show the trajectories typical for finite/fixed-time convergent systems. 
To illustrate the quantization effect, Fig  \ref{fig:4} shows the evolution of the quantized measurements.

%The simulation results in Fig. \ref{fig:2} demonstrate the effectiveness of stabilizing a nilpotent $A=A_1$. In this scenario, the diagonal $G_\dn$ reduces the quantization function to a logarithmic quantizer elementwise. This validates the efficacy of finite/fixed-time stabilization using homogeneous quantized state feedback with a logarithmic quantizer. Furthermore, the nearly fixed-time feedback exhibits faster convergence in driving the system into a neighborhood of the origin. However, when the state is closer to the origin, the nearly fixed-time feedback shows slower convergence compared to finite-time feedback.

%For a more general linear system $(A_2,B_2)$, we present the results of switch homogeneous feedback, as proposed in Corollary \ref{cor:comu}, in Fig. \ref{fig:3}. As shown, both controllers effectively drive the system to the origin within a bounded settling time, with fixed-time stabilization showing faster convergence than finite-time stabilization.

In Fig. \ref{fig:4}, we provide the quantized signal plots, offering insight into the system's behavior under the implemented quantization scheme.

\begin{figure}[htbp]
    \centering
\includegraphics[width = 0.45\textwidth]{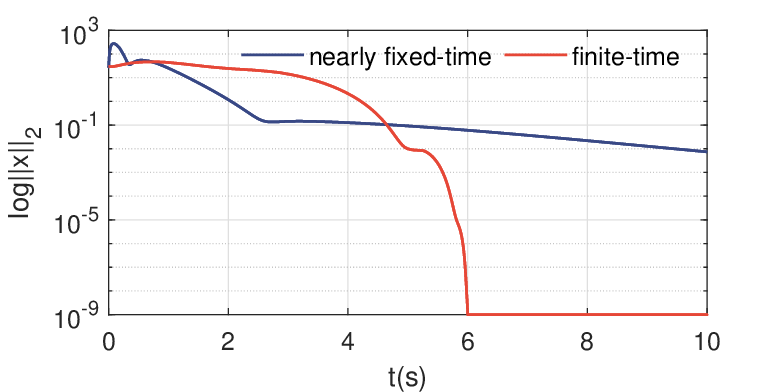}
    \caption{The closed-loop system
\eqref{eq:lin_sys}, \eqref{eq:u_q} with $A\!=\!A_1,B\!=\!B_1$.}
    \label{fig:2}
\end{figure}

\begin{figure}[htbp]
    \centering
    \includegraphics[width = 0.45\textwidth]{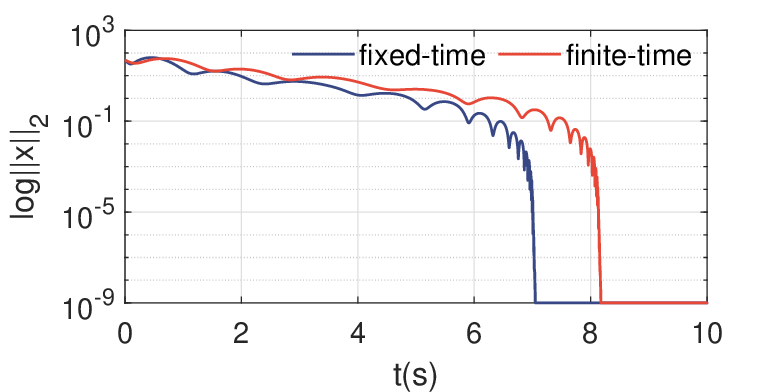}
 \caption{The closed-loop system \eqref{eq:q_sys}, \eqref{eq:u_c} with $A = A_2$, $B=B_2$. }
    \label{fig:3}
\end{figure}
\begin{figure}[htbp]
    \centering
    \includegraphics[width = 0.45\textwidth]{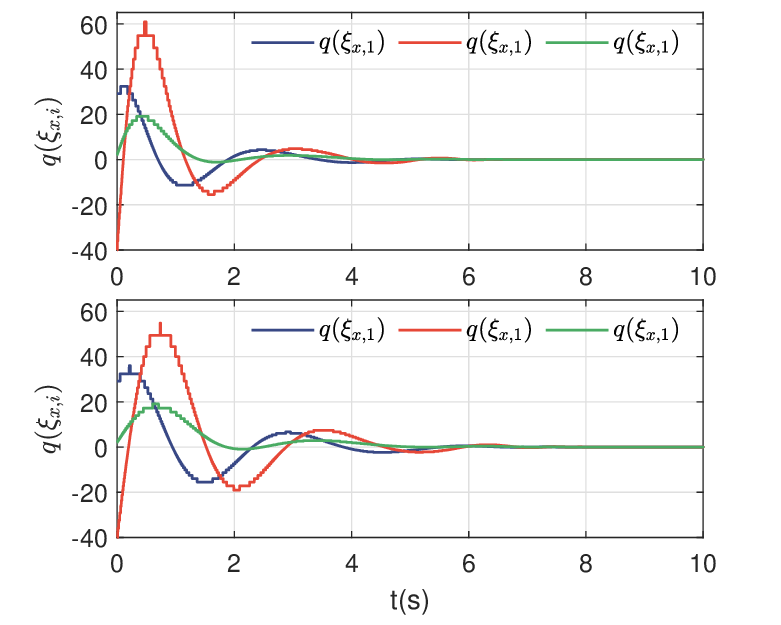}
 \caption{Quantized measurements of $\Gamma^{-1}x$ using a logarithmic quantizer for system \eqref{eq:q_sys} and \eqref{eq:u_c} with $A = A_2$ and $B = B_2$: upper figure for the fixed-time control and lower figure for finite-time control.}
    \label{fig:4}
\end{figure}

\section{Conclusion}
This paper addressed the problem of finite/fixed-time stabilization of LTI systems under quantized state measurements. It is discovered that the  stability and convergence rates of the LTI systems with a homogeneous control can be preserved at least locally under certain logarithmic quantizer. Global finite/fixed-time stability is proved for homogeneous plants (i.e., for LTI systems with nilpotent matrices).
For a general LTI plant, a global finite/fixed-time stabilization with quantized measurements can be achieved using a switched locally homogeneous control.
The theoretical results are supported by numerical simulations. 
By providing valuable insights into the behavior of homogeneous systems with  quantization effects, these results contribute to the understanding and applicability of finite/fixed-time controller under quantized state measurements. Future research will focus on investigating robustness and exploring a new kind of quantizers (e.g., inspired by the homogeneity).

%In linear dynamic stability analysis, the system and the quadratic Lyapunov function are typically defined in a metric space. However, for nonlinear system stability analysis, the Lyapunov function may introduce a topology in $\mathbb{R}^n$. Specifically, for the system \eqref{eq:lin_sys}, \eqref{eq:hom_con_P}, the canonical homogeneous norm serves as a Lyapunov function. The canonical homogeneous norm $\|\cdot\|_\dn$ is not a norm in the conventional sense since it does not satisfy the triangle inequality. Nevertheless, it defines a norm topology in $\mathbb{R}^n$.Quantizers commonly involve partitioning a Euclidean space. To bridge the gap between stability analysis and quantizer design, it is important to establish a relationship between the metric $|\cdot|$ and the topology induced by the homogeneous norm $\|\cdot\|_\dn$. In the subsequent analysis, we derive essential property that establishs this connection.

\bibliographystyle{unsrt}
\bibliography{reference}

\begin{thebibliography}{10}

\bibitem{hristu2005handbook}
D.~Hristu-Varsakelis and W.~S Levine.
\newblock {\em Handbook of networked and embedded control systems}.
\newblock Number TK7895. E42. H29 2005. Springer, 2005.

\bibitem{gray1998quantization}
R.~M. Gray and D.~L. Neuhoff.
\newblock Quantization.
\newblock {\em IEEE Transactions on Information Theory}, 44(6):2325--2383,
  1998.

\bibitem{bhat_etal_2005_MCSS}
S.~P Bhat and D.~S Bernstein.
\newblock Geometric homogeneity with applications to finite-time stability.
\newblock {\em Mathematics of Control, Signals and Systems}, 17(2):101--127,
  2005.

\bibitem{Orlov2004:SIAM}
Y.~Orlov.
\newblock Finite time stability and robust control synthesis of uncertain
  switched systems.
\newblock {\em SIAM Journal on Control and Optimization}, 43(4):1253--1271,
  2004.

\bibitem{hong2006finite}
Y~Hong and Z-P Jiang.
\newblock Finite-time stabilization of nonlinear systems with parametric and
  dynamic uncertainties.
\newblock {\em IEEE Transactions on Automatic Control}, 51(12):1950--1956,
  2006.

\bibitem{hong2010finite}
Y.~Hong, Z-P Jiang, and G.~Feng.
\newblock Finite-time input-to-state stability and applications to finite-time
  control design.
\newblock {\em SIAM Journal on Control and Optimization}, 48(7):4395--4418,
  2010.

\bibitem{PolyakovTAC2012}
A.~Polyakov.
\newblock Nonlinear feedback design for fixed-time stabilization of linear
  control systems.
\newblock {\em IEEE Transactions on Automatic Control}, 57(8):2106--2110, 2012.

\bibitem{Song_etal2017:Aut}
Y.~Song, Y.~Wang, J.~Holloway, and M.~Krstic.
\newblock Time-varying feedback for regulation of normal-form nonlinear systems
  in prescribed finite time.
\newblock {\em Automatica}, 83:243--251, 2017.

\bibitem{andrieu2008homogeneous}
V.~Andrieu, L.~Praly, and A.~Astolfi.
\newblock Homogeneous approximation, recursive observer design, and output
  feedback.
\newblock {\em SIAM Journal on Control and Optimization}, 47(4):1814--1850,
  2008.

\bibitem{polyakov2020book}
A.~Polyakov.
\newblock {\em Generalized homogeneity in systems and control}.
\newblock Springer, 2020.

\bibitem{zubov1958systems}
V.~I. Zubov.
\newblock Systems of ordinary differential equations with
  generalized-homogeneous right-hand sides.
\newblock {\em Izvestiya Vysshikh Uchebnykh Zavedenii. Matematika}, (1):80--88,
  1958.

\bibitem{rosier1992SCL}
L.~Rosier.
\newblock Homogeneous lyapunov function for homogeneous continuous vector
  field.
\newblock {\em Systems \& Control Letters}, 19(6):467--473, 1992.

\bibitem{kawski1990_CTAT}
M.~Kawski.
\newblock Homogeneous stabilizing feedback laws.
\newblock {\em Control Theory and advanced technology}, 6(4):497--516, 1990.

\bibitem{grune2000SIAM}
L.~Gr{\"u}ne.
\newblock Homogeneous state feedback stabilization of homogenous systems.
\newblock {\em SIAM Journal on Control and Optimization}, 38(4):1288--1308,
  2000.

\bibitem{hong2001output}
Y.~Hong, J.~Huang, and Y.~Xu.
\newblock On an output feedback finite-time stabilization problem.
\newblock {\em IEEE Transactions on Automatic Control}, 46(2):305--309, 2001.

\bibitem{nakamura_etal_2009_TAC}
N.~Nakamura, H.~Nakamura, Y.~Yamashita, and H.~Nishitani.
\newblock Homogeneous stabilization for input affine homogeneous systems.
\newblock {\em IEEE Transactions on Automatic Control}, 54(9):2271--2275, 2009.

\bibitem{polyakov2019IJRNC}
A.~Polyakov.
\newblock Sliding mode control design using canonical homogeneous norm.
\newblock {\em International Journal of Robust and Nonlinear Control},
  29(3):682--701, 2019.

\bibitem{zimenko_etal_2020_TAC}
K.~Zimenko, A.~Polyakov, D.~Efimov, and W.~Perruquetti.
\newblock Robust feedback stabilization of linear mimo systems using
  generalized homogenization.
\newblock {\em IEEE Transactions on Automatic Control}, 65(12):5429--5436,
  2020.

\bibitem{AcaryBrogliato2010:SCL}
V.~Acary and B.~Brogliato.
\newblock Implicit euler numerical scheme and chattering-free implementation of
  sliding mode systems.
\newblock {\em Systems \& Control Letters}, 59(5):284--293, 2010.

\bibitem{bernuau_etal2017Aut}
E.~Bernuau, E.~Moulay, and P.~Coirault.
\newblock Stability of homogeneous nonlinear systems with sampled-data inputs.
\newblock {\em Automatica}, 85:349--355, 2017.

\bibitem{polyakov_etal2023:Aut}
A.~Polyakov, D.~Efimov, and X.~Ping.
\newblock Consistent discretization of homogeneous finite/fixed-time
  controllers for lti systems.
\newblock {\em Automatica}, 155:111118, 2023.

\bibitem{Fu_etal_2005_TAC}
M.~Fu and L.~Xie.
\newblock The sector bound approach to quantized feedback control.
\newblock {\em IEEE Transactions on Automatic control}, 50(11):1698--1711,
  2005.

\bibitem{delchamps1990TAC}
D.~F Delchamps.
\newblock Stabilizing a linear system with quantized state feedback.
\newblock {\em IEEE Transactions on Automatic Control}, 35(8):916--924, 1990.

\bibitem{brockett_Liberzon2000TAC}
R.~W Brockett and D.~Liberzon.
\newblock Quantized feedback stabilization of linear systems.
\newblock {\em IEEE Transactions on Automatic Control}, 45(7):1279--1289, 2000.

\bibitem{gao2008Aut}
H.~Gao and T.~Chen.
\newblock A new approach to quantized feedback control systems.
\newblock {\em Automatica}, 44(2):534--542, 2008.

\bibitem{kang2015Aut}
X.~Kang and H.~Ishii.
\newblock Coarsest quantization for networked control of uncertain linear
  systems.
\newblock {\em Automatica}, 51:1--8, 2015.

\bibitem{su2018Aut}
L.~Su and G.~Chesi.
\newblock Robust stability of uncertain linear systems with input and output
  quantization and packet loss.
\newblock {\em Automatica}, 87:267--273, 2018.

\bibitem{liberzon2007TAC}
D.~Liberzon and D.~Nesic.
\newblock Input-to-state stabilization of linear systems with quantized state
  measurements.
\newblock {\em IEEE Transactions on Automatic Control}, 52(5):767--781, 2007.

\bibitem{elia_etal_2001_TAC}
N.~Elia and S.~K Mitter.
\newblock Stabilization of linear systems with limited information.
\newblock {\em IEEE Transactions on Automatic Control}, 46(9):1384--1400, 2001.

\bibitem{bullo_Liberzon2006TAC}
F.~Bullo and D.~Liberzon.
\newblock Quantized control via locational optimization.
\newblock {\em IEEE Transactions on Automatic Control}, 51(1):2--13, 2006.

\bibitem{wang2021TAC}
J.~Wang.
\newblock Quantized feedback control based on spherical polar coordinate
  quantizer.
\newblock {\em IEEE Transactions on Automatic Control}, 66(12):6077--6084,
  2021.

\bibitem{wang2022TAC}
J.~Wang.
\newblock Quantized feedback control of discrete-time mimo linear systems with
  input and output quantization over finite data rate channels.
\newblock {\em IEEE Transactions on Automatic Control}, 2022.

\bibitem{nesic2009TAC}
D.~Nesic and D.~Liberzon.
\newblock A unified framework for design and analysis of networked and
  quantized control systems.
\newblock {\em IEEE Transactions on Automatic control}, 54(4):732--747, 2009.

\bibitem{ceragioli2007CDC}
F.~Ceragioli and C.~De~P.
\newblock Discontinuous stabilization of nonlinear systems: Quantized and
  switching controls.
\newblock {\em Systems \& control letters}, 56(7-8):461--473, 2007.

\bibitem{liu2012Aut}
T.~Liu, Z-P Jiang, and D.~J Hill.
\newblock A sector bound approach to feedback control of nonlinear systems with
  state quantization.
\newblock {\em Automatica}, 48(1):145--152, 2012.

\bibitem{rehak2019IFAC}
B.~Reh{\'a}k and V.~Lynnyk.
\newblock Decentralized networked stabilization of a nonlinear large system
  under quantization.
\newblock {\em IFAC-PapersOnLine}, 52(20):49--54, 2019.

\bibitem{yan2019Aut}
Y.~Yan, S.~Yu, and X.~Yu.
\newblock Quantized super-twisting algorithm based sliding mode control.
\newblock {\em Automatica}, 105:43--48, 2019.

\bibitem{ye2022TAC}
Z.~Ye, D.~Zhang, J.~Cheng, and Z.~Wu.
\newblock Event-triggering and quantized sliding mode control of umv systems
  under dos attack.
\newblock {\em IEEE Transactions on Vehicular Technology}, 71(8):8199--8211,
  2022.

\bibitem{zhou_etal2023:CDC}
Y.~Zhou, A.~Polyakov, and G.~Zheng.
\newblock Homogeneous finite/fixed-time stabilization with quantization.
\newblock In {\em 62th IEEE Conference on Decision and Control (CDC)}. IEEE,
  2023.

\bibitem{boots2009spatial}
B.~Boots, K.~Sugihara, S.~N. Chiu, and A.~Okabe.
\newblock Spatial tessellations: concepts and applications of voronoi diagrams.
\newblock 2009.

\bibitem{filippov2013differential}
A.~F. Filippov.
\newblock {\em Differential equations with discontinuous righthand sides:
  control systems}, volume~18.
\newblock Springer Science \& Business Media, 2013.

\bibitem{khomenyuk1961systems}
VV~Khomenyuk.
\newblock Systems of ordinary differential equations with generalized
  homogeneous right-hand sides.
\newblock {\em Izvestiya Vysshikh Uchebnykh Zavedenii. Matematika},
  (3):157--164, 1961.

\bibitem{kawski1991}
M.~Kawski.
\newblock Families of dilations and asymptotic stability.
\newblock In {\em Analysis of controlled dynamical systems}, pages 285--294.
  Springer, 1991.

\bibitem{kawski1995IFAC}
M.~Kawski.
\newblock Geometric homogeneity and stabilization.
\newblock {\em IFAC Proceedings Volumes}, 28(14):147--152, 1995.

\bibitem{polyakov2016IJRNC}
A.~Polyakov, D.~Efimov, and W.~Perruquetti.
\newblock Robust stabilization of mimo systems in finite/fixed time.
\newblock {\em International Journal of Robust and Nonlinear Control},
  26(1):69--90, 2016.

\end{thebibliography}

\end{document}